\documentclass{amsart}
\usepackage{amsfonts}
\usepackage{graphicx}
\usepackage{tabularx}

\usepackage[usenames,dvipsnames]{color}
\usepackage{comment}

\begin{document}

\newcommand{\C}{\mathbb{C}}
\newcommand{\F}{\mathbb{F}}
\newcommand{\R}{\mathbb{R}}
\newcommand{\Q}{\mathbb{Q}}
\newcommand{\N}{\mathbb{N}}
\newcommand{\Z}{\mathbb{Z}}
\newcommand{\PP}{\mathbb{P}}
\newcommand{\HH}{\mathbb{H}}
\newcommand{\MU}{\mathbb{\mu}}
\newcommand{\TT}{\mathbb{T}}
\newcommand{\GG}{\mathbb{G}}
\newcommand{\SSS}{\mathbb{S}}
\newcommand{\m}{\mathrm{m}}

\newtheorem{thm}{Theorem}
\newtheorem{defn}[thm]{Definition}
\newtheorem{prop}[thm]{Proposition}
\newtheorem{coro}[thm]{Corollary}
\newtheorem{lem}[thm]{Lemma}
\newtheorem{conj}[thm]{Conjecture}
\newtheorem{rem}[thm]{Remark}
\newtheorem{ex}[thm]{Example}
\newtheorem{exs}[thm]{Examples}
\newcommand{\dd}{\;\mathrm{d}}
\newcommand{\e}{\mathrm{e}}
\newcommand{\pf}{\noindent {\bf PROOF.} \quad}
\newcommand{\res}{\mathrm{Res}}

\newcommand{\Li}{\mathrm{Li}} 
\newcommand{\I}{\mathrm{I}} 
\newcommand{\re}{\mathop{\mathrm{Re}}} 
\newcommand{\im}{\mathop{\mathrm{Im}}} 
\newcommand{\ii}{\mathrm{i}} 
\newcommand{\Lf}{\mathrm{L}} 
\newtheorem{obs}[thm]{Observation} 

\newcommand{\tushant}[1]{{\color{blue}{Tushant: #1}} }
\newcommand{\matilde}[1]{{\color{red}{Matilde: #1}} }

\title{The Mahler measure for arbitrary tori}
\author{Matilde Lal\'in}
\address{D\'{e}partment de Math\'{e}matique et de Statistique, Universit\'{e} de Montr\'{e}al. CP 6128, succ. Centre-Ville. Montreal, QC H3C 3J7, Canada}
\email{mlalin@dms.umontreal.ca}

\author{Tushant Mittal}
\address{Department of Computer Science and Engineering, Indian Institute of Technology Kanpur, India}
\email{mittaltushant@gmail.com}

\thanks{This research was supported by the Natural Sciences and Engineering Research Council
of Canada [Discovery Grant 355412-2013 to ML] and Mitacs [Globalink Research Internship to TM]}


\begin{abstract} We consider a variation of the Mahler measure where the defining integral is performed over a more general torus. We focus our investigation  on two particular polynomials related to certain elliptic curve $E$ and we establish new formulas
for this variation of the Mahler measure in terms of $L'(E,0)$.
\end{abstract}

\maketitle

\section{Introduction}

The Mahler measure of a non-zero rational function $P \in \C(x_1,\dots,x_n)$ is defined by 
\begin{equation*}
 \m(P):=\frac{1}{(2\pi i)^n}\int_{\mathbb{T}^n}\log|P(x_1,\dots, x_n)|\frac{dx_1}{x_1}\cdots \frac{dx_n}{x_n},
\end{equation*}
where $\mathbb{T}^n=\{(x_1,\dots,x_n)\in \mathbb{C}^n : |x_1|=\cdots=|x_n|=1\}$. 

This construction, when applied to multivariable polynomials, often yields values of special functions, including some of number theoretic interest, such as 
the Riemann zeta-function and $L$-functions associated to arithmetic-geometric objects such as elliptic curves. 

The initial formulas for multivariable polynomials were due to Smyth \cite{S1,S1a}. These early results included
\[\m(x+y+1)= \frac{3 \sqrt{3}}{4 \pi} L(\chi_{-3},2)= L '(\chi_{-3}, -1),\]
where
\[ L(\chi_{-3}, s) = \sum_{n=1}^\infty \frac{\chi_{-3}(n)}{n^s}
\quad \mbox{and} \quad \chi_{-3}(n) = \left \{ \begin{array}{rl} 1 & \mbox{if}\quad n \equiv1 \; \mod \; 3, \\ -1 & \mbox{if}\quad n \equiv -1\;  \mod \; 3, \\ 0 & \mbox{if}\quad n \equiv 0 \; \mod \; 3, \end{array} \right. \]
is a Dirichlet $L$-function. 

The formula above was generalized by Cassaigne and Maillot \cite{CM} in the following fashion. Let $a, b, c$ be nonzero complex numbers. Then 
\begin{equation} \label{eq:Maillot}
\pi \m(ax+by+c) = \left \{ \begin{array}{lr} \alpha \log |a| + \beta \log |b| +\gamma \log |c| + D\left(\left|\frac{a}{b}\right| 
e^{i \gamma}\right) & \quad \triangle,\\\\ \pi \log \max \{ |a|, |b|, |c|\} & \quad \mbox{not}\, \triangle, \end{array} \right. 
\end{equation}
where $\triangle$ stands for the statement that $|a|$, $|b|$, and
$|c|$ are the lengths of the sides of a triangle; and $\alpha$,
$\beta$, and $\gamma$ are the angles opposite to the sides of lengths
$|a|$, $|b|$ and $|c|$ respectively (Figure \ref{figmaillot}). The function $D$ is the Bloch--Wigner dilogarithm (to be defined in Section \ref{sec:regulator}, equation \eqref{Bloch-Wigner}) and the corresponding term  then codifies the volume of an ideal  hyperbolic tetrahedron in $\mathbb{H}^3 \cong \C \times \R_{\geq 0}$ with basis the triangle whose sides are $|a|$, $|b|$, and $|c|$ and fourth vertex infinity.
\begin{figure}
\includegraphics[width=18pc]{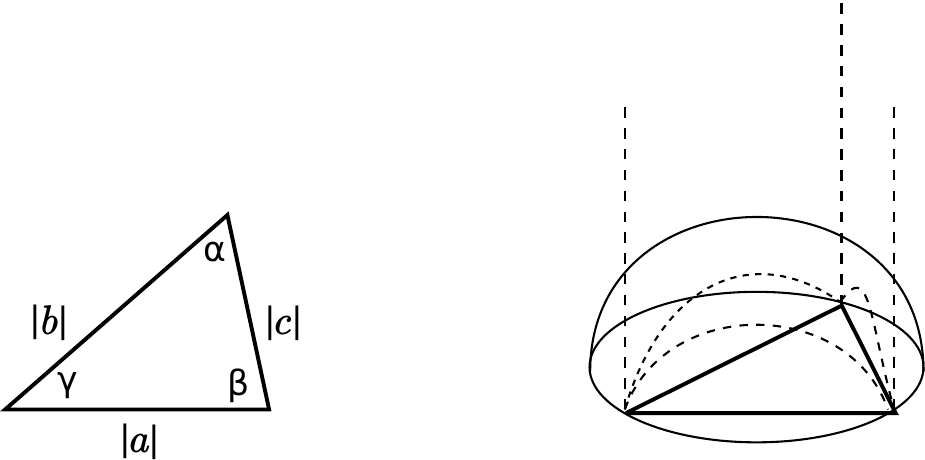}
\caption{Relation among the parameters in Maillot's formula.}
\label{figmaillot}
\end{figure}

We remark that the formulas above also apply when the constant coefficient is replaced by a variable, in the sense that $\m(ax+by+cz)=\m(ax+by+c)$.

The connection of Mahler measure with elliptic curves was predicted by Deninger \cite{Deninger}, and examined in detail by Boyd \cite{Bo98} and Rodriguez-Villegas \cite{RV}. Boyd computed numerical examples very systematically for several families of polynomials. For example, he considered the families
\begin{align*}
R_\alpha(x,y):=&(1+x)(1+y)(x+y)-\alpha xy,\\
S_{k,\beta}(x,y)=& y^2+kxy-x^3-\beta x,
\end{align*}
where $\alpha, k, \beta$ are integral parameters. Boyd found that for $|\alpha|\leq 100$, and for many cases of $\beta, k${\footnote{Boyd gave precise conjectural conditions for $\beta$ and $k$.}},
\begin{align*}
\m(R_\alpha(x,y))\stackrel{?}{=}&r_\alpha L'(E_{N(\alpha)},0),\\
\m(S_{k,\beta}(x,y))\stackrel{?}{=}& \frac{1}{4}\log|\beta|+s_{k,b}L'(E_{N(k,\beta)},0),
\end{align*}
where $r_\alpha$, $s_{k,\beta}$  are rational numbers of low height, the $L$-functions are attached to elliptic curves that 
are defined by $R_\alpha(x,y)=0$, and $S_{k,\beta}(x,y)=0$ respectively, and the question mark stands for a numerical formula that is true for at least 20 decimal places.
In all cases, $N$ denotes the conductor of the elliptic curve, which is a function on the coefficients of the corresponding polynomial.

Some of Boyd's conjectures have been proven. In particular, Rogers and Zudilin \cite{RZ12} proved that
\begin{align}\label{eq:Rogers-Zudilin}
\m(R_{-2}(x,y))=&\frac{15}{\pi^2}L(E_{20},2)=3 L'(E_{20},0),\\
\m(R_{4}(x,y))= &\frac{10}{\pi^2}L(E_{20},2)=2 L'(E_{20},0). \nonumber
\end{align}
Mellit \cite{Me12} proved similar formulas for $\alpha=1,7,-8$, which correspond to the conductor 14 case. 

On the other hand, the Mahler measure of $S_{2,-1}(x,y)$, which is the Weierstrass form of an elliptic curve of conductor 20,
was studied by Touafek \cite{Touafek, Touafek-thesis}, who exhibited an argument that leads to 
\[\m(S_{2,-1}(x,y))=\frac{2}{3}\m(R_{-2}(x,y)),\]
provided that one properly establishes a relationship between certain cycles in $H_1(E,\Z)$. (More details of this argument
were completed by Bertin \cite{Bertin-banff} who developed this idea to give an alternative proof of equation \eqref{eq:Rogers-Zudilin}.)

In this work, we consider the following extension of Mahler measure. 
\begin{defn} Let $a_1,\dots,a_n \in \R_{>0}$. The $(a_1,\dots, a_n)$-Mahler measure of a non-zero rational function $P \in \C(x_1,\dots,x_n)$ is defined by 
\begin{equation*}
 \m_{a_1,\dots, a_n}(P):=\frac{1}{(2\pi i)^n}\int_{\mathbb{T}_{a_1}\times \dots \times\mathbb{T}_{a_n}}\log|P(x_1,\dots, x_n)|\frac{dx_1}{x_1}\cdots \frac{dx_n}{x_n},
\end{equation*}
where $\mathbb{T}_a=\{x\in \mathbb{C} : |x|=a\}$. 
\end{defn}
This idea of considering arbitrary tori in the integration was initially proposed to us by Rodriguez-Villegas a long time ago. 

Given this definition, Cassaigne and Maillot's formula can be interpreted as $\m_{a,b,c}(x+y+z)$. Some cases of the  formula of Vandervelde \cite{Vandervelde} for the Mahler measure of $axy+bx+cy+d$ may be also viewed in this context.  

Our goal is to explore this definition for other formulas, in particular those involving elliptic curves. More precisely, we connect this idea with Boyd's examples in order to prove the following results.

\begin{thm}\label{bigthm} 
\begin{align*}
\m_{a,a}(y^2+2xy-x^3+x)&=\left \{\begin{array}{ll}
2 \log a+ 2 L'(E_{20},0) & \frac{\sqrt{5}-1}{2} \leq a \leq \frac{1+\sqrt{5}}{2}\\\\
                                 3 \log a & a \geq \frac{3+\sqrt{13}}{2},\\
                                 \log a &  0< a\leq \frac{-3+\sqrt{13}}{2}.
                                 \end{array}
 \right.\\ 
\m_{a^2,a}((1+x)(1+y)(x+y)+2 xy)&= \left \{\begin{array}{ll}
4 \log a  + 3 L'(E_{20},0) & 1\leq a\leq \sqrt{\frac{1+\sqrt{5}+\sqrt{2\sqrt{5}+2}}{2}},\\
2 \log a  + 3 L'(E_{20},0) & \sqrt{\frac{1+\sqrt{5}-\sqrt{2\sqrt{5}+2}}{2}}\leq a \leq 1.
\end{array}\right.
\end{align*}
\end{thm}

The results above are, in some sense, quite restricted, due to the technical difficulties involving the study of the integration path. They are similar in generality to some earlier formulas from \cite{L} 
 that involve a single varying parameter and relate Mahler measures to polylogarithms. By changing variables $y \rightarrow ay$ and $x\rightarrow ax$ and dividing by $a^2$ in the first polynomial, and $y \rightarrow ay$ and $x\rightarrow a^2x$ and dividing by $a$ in the second polynomial, 
we obtain the following corollary which expresses the same results in terms of the classical Mahler measure of non-tempered polynomials (as defined in Section\ref{sec:regulator}, Definition \ref{tempered}), which are very interesting in their own, since the $K$-theory framework does not completely apply to these cases. 
\begin{coro}
\begin{align*}
 \m(y^2+2xy-ax^3+a^{-1}x)&=\left \{\begin{array}{ll}
2 L'(E_{20},0)& \frac{\sqrt{5}-1}{2} \leq a \leq \frac{1+\sqrt{5}}{2}\\\\
                                 \log a & a \geq \frac{3+\sqrt{13}}{2},\\
                                 \log \left(a^{-1}\right) &  0< a\leq \frac{-3+\sqrt{13}}{2}.
                                 \end{array}
 \right.\\
\m((1+a^2x)(1+ay)(ax+y)+2 a^2xy)&= \left \{\begin{array}{ll}
3 \log a  + 3 L'(E_{20},0) & 1\leq a\leq \sqrt{\frac{1+\sqrt{5}+\sqrt{2\sqrt{5}+2}}{2}},\\
 \log a  + 3 L'(E_{20},0) & \sqrt{\frac{1+\sqrt{5}-\sqrt{2\sqrt{5}+2}}{2}}\leq a \leq 1.
\end{array}\right.
 \end{align*}
 \end{coro}

Our method of proof follows several steps. In Section \ref{sec:regulator} we recall the relationship between the Mahler measure of polynomials associated to elliptic curves and the regulator. 
Then, in Section \ref{sec:arbitrary-torus} we analyze what happens to the regulator integral when the integration domain is changed according to our new definition. We start working with our particular 
examples in Section \ref{sec:regulators-examples}, where we establish the relationship between the regulators of these two families. After that, it remains to discuss the integration paths and to characterize them as elements in the homology,
which is done in Sections \ref{sec:path} and \ref{sec:cycle}. Section \ref{sec:arg} deals with a technical residue computation. The proof of our result is completed in Section \ref{sec:proof}.

\section{The connection between Mahler measure and the regulator}\label{sec:regulator}

In this section we recall the definition of the regulator on the second $K$-group of an elliptic curve $E$ given by Bloch  and Be\u\i linson  
and explain how it can be computed in terms of the elliptic dilogarithm. We then discuss the relationship between Mahler measure and the regulator. 

Let $F$ be a field. Thanks to Matsumoto's theorem, the second $K$-group of $F$ can be described as
\[K_2(F) \cong \Lambda^2 F^\times/\{x\otimes (1-x): x \in F, x\not = 0,1\}.\]

\begin{defn}\label{tempered} Let $P(x,y)\in \C[x^\pm,y^\pm]$ be a two variable polynomial. The Newton polygon $\Delta(P)\subset \R^2$ is given by
\[\mbox{the convex hull of }\{(m,n)\in \Z^2\, |\, \mbox{ the coefficient of }x^my^n \mbox{ of } P \mbox{ is non-zero.}\}\]
For each side of $\Delta$, one can associate a one-variable polynomial whose coefficients are the coefficients of $P$ associated to the points that lie on that side. 

$P$ is said to be tempered if the Mahler measures of its side polynomials are zero. (See Section III of \cite{RV} for more details on these definitions.)
\end{defn}

Let $E/\Q$ be an elliptic curve given by an equation $P(x,y)=0$. 
Rodriguez-Villegas \cite{RV} associates the condition that the polynomial $P$ is tempered to the conditions that the 
tame symbols in $K$-theory are trivial. In that case we can think of  $K_2(E)\otimes \Q \subset K_2(\Q(E)) \otimes \Q$.

Let $x,y \in \Q(E)$. We will work with the differential form 
\begin{equation}\label{eq:eta}
\eta(x,y) := \log |x| d \arg y - \log|y| d \arg x,
\end{equation}
where $d \arg x$ is defined by $\im(dx/x)$. 

The Bloch--Wigner dilogarithm is given by 
\begin{equation}\label{Bloch-Wigner}
D(x)= \im(\Li_2(x))+\arg(1-x) \log|x|,
\end{equation}
where
\[\Li_2(x)=-\int_0^x\frac{\log(1-z)}{z}dz.\]

The form $\eta(x,y)$ is multiplicative, antisymmetric, and satisfies
\[\eta(x,1-x) = d D(x).\]

We are now ready to define the regulator. 

\begin{defn}The regulator map of Bloch \cite{Bloch} and Be\u\i linson \cite{Beilinson} is given by
\begin{eqnarray*}
r_E:K_2(E)\otimes \Q &\rightarrow& H^1(E,\mathbb{R})\\
\{x,y\}&\rightarrow &\left\{ [\gamma] \rightarrow \int_\gamma \eta(x,y)\right \}.
 \end{eqnarray*}
\end{defn} 
In the above definition, we take $[\gamma] \in H_1(E,\mathbb{Z})$ and interpret $H^1(E,\mathbb{R})$ as the dual of $H_1(E,\mathbb{Z})$. 

We remark that the regulator is actually defined over $K_2(\mathcal{E})$, where $\mathcal{E}$ is the N\'eron model of the elliptic curve. $K_2(\mathcal{E})\otimes \Q$ is a subgroup of $K_2(E)\otimes \Q$ determined by finitely many extra conditions as described in \cite{BG}. 

\begin{rem}\label{remark}
Due to the action of complex conjugation on $\eta$, the regulator map is trivial for the classes that remain invariant by complex conjugation, denoted by  $H_1(E,\mathbb{Z})^+$. It therefore suffices to consider the regulator as a function on $H_1(E,\mathbb{Z})^-$.
\end{rem}

We proceed to discuss the integral of $\eta(x,y)$. Since $E/\Q$ is an elliptic curve, we can write
\begin{equation}
\label{eq:diagrammaps}\begin{array}{ccccc}
E(\C)& \stackrel{\sim}{\rightarrow} &\C/(\Z+\tau\Z) &\stackrel{\sim}{\rightarrow}&\C^\times /q^\Z\\ \\
P=(\wp(u),\wp'(u))& \rightarrow & u \bmod \Lambda & \rightarrow & z=e^{2\pi i u},
\end{array}
\end{equation}
where $\wp$ is the Weierstrass function, $\Lambda$ is the lattice $\Z+\tau\Z$, $\tau \in \HH$, and $q=e^{2 \pi i \tau}$.

The next definition is due to Bloch \cite{Bloch}.
\begin{defn}
The elliptic dilogarithm is a function on $E(\C)$ given for $P\in E(\C)$ corresponding to $z \in \C^\times/q^\Z$ by 
\begin{equation}\label{eq:elldilogdef}
D^E(P):=\sum_{n \in \Z} D(q^nz),
\end{equation}
where $D$ is the Bloch--Wigner dilogarithm defined by \eqref{Bloch-Wigner}.
\end{defn}

Let $\Z[E(\C)]$ be the group of divisors on $E$ and let
\[\Z[E(\C)]^-\cong \Z[E(\C)]/\{ (P)+(-P): P\in E(\C)\}.\]

Let $x,y \in \C(E)^\times$. We define a diamond operation by 
\begin{eqnarray*}
\diamond: \Lambda^2 \C(E)^\times &\rightarrow &\Z[E(\C)]^-\\
(x)\diamond (y) &=& \sum_{i,j}m_in_j (S_i-T_j),
\end{eqnarray*}
where \[(x)=\sum_i m_i (S_i) \mbox{ and  }(y)=\sum_j n_j (T_j).\]

With these elements, we have the following result.
\begin{thm}\label{thm:bloch}(Bloch \cite{Bloch}) The elliptic dilogarithm $D^E$ extends by linearity to a map from $\Z[E(\Q)]^-$ to $\C$. Let $x, y \in \Q(E)$ and $\{x,y\}\in K_2(E)$. Then
\[r_E(\{x,y\})[\gamma]=D^E((x)\diamond(y)),\]
where $[\gamma]$ is a generator of $H_1(E,\Z)^-$.
 \end{thm}

 Deninger \cite{Deninger} was the first to write a formula of the form
\begin{equation}\label{eq:mregulator}
 \m(P) = \frac{1}{2\pi}r(\{x,y\})[\gamma].
 \end{equation}
Rodriguez-Villegas \cite{RV} made a thorough study of the properties of $\eta(x,y)$, defining the notion of tempered polynomial to characterize those polynomials that fit in the above picture.
He also combined the above expression with Theorem \ref{thm:bloch} to prove an identity between two Mahler measures (originally conjectured in Boyd \cite{Bo98})
in \cite{RV2}.

We will now discuss how to reach formula \eqref{eq:mregulator} in a concrete two-variable polynomial.
Let $P(x,y) \in \C[x,y]$ be a polynomial of degree 2 on $y$. 
We may then write 
\[P(x,y)=P^*(x)(y-y_1(x))(y-y_2(x)),\]
where $y_1(x), y_2(x)$ are algebraic functions. 

By applying Jensen's formula with respect to the variable $y$, we have
\begin{align*}
\m(P)-\m(P^*)=&\frac{1}{(2 \pi i)^2} \int_{\TT^2}\log |P(x,y)|\frac{dx}{x}\frac{dy}{y}-\m(P^*)\\
=& \frac{1}{(2 \pi i)^2} \int_{\TT^2}(\log |y-y_1(x)|+\log|y-y_2(x)|)\frac{dx}{x}\frac{dy}{y}\\
=& \frac{1}{2\pi i} \int_{|x|=1,|y_1(x)|\geq 1} \log|y_1(x)|\frac{dx}{x}+\frac{1}{2\pi i} \int_{|x|=1,|y_2(x)|\geq 1} \log|y_2(x)|\frac{dx}{x}.
\end{align*}
Recalling formula \eqref{eq:eta} for $\eta(x,y)$, we have,
\begin{align*}
\m(P)-\m(P^*) =& -\frac{1}{2 \pi} \int_{|x|=1,|y_1(x)|\geq 1} \eta(x,y_1)-\frac{1}{2 \pi} \int_{|x|=1,|y_2(x)|\geq 1} \eta(x,y_2).
\end{align*}

Sometimes we will encounter the case that one of the roots $y_1(x)$ has always absolute value greater than or equal to 1 as $|x|=1$ 
and the other root has always absolute value smaller than or equal to 1 as $|x|=1$. This will allow us to write the right-hand side as a single term,
an integral over a closed path. 

When $P$ corresponds to an elliptic curve and when the set $\{|x|=1,|y_i(x)|\geq 1\}$ can be seen as a cycle in $H_1(E,\Z)^-$, then we may be able to recover a formula of the type \eqref{eq:mregulator}. This has to be examined on a case by case basis.

\section{The initial analysis with an arbitrary torus} \label{sec:arbitrary-torus}

When working with an arbitrary torus, we can still do a similar analysis to the one in the previous section. 
We continue with the notation that $P(x,y) \in \C[x,y]$ is a polynomial of degree 2 on $y$. 

Let $x=a\tilde{x}$ and $y=b\tilde{y}$. We have
\begin{align*}
\m_{a,b}(P)-\m_{a,b}(P^*)=&\frac{1}{(2\pi i)^2} \int_{|x|=a, |y|=b} \log|P(x,y)|\frac{dx}{x}\frac{dy}{y}-\m_{a,b}(P^*)\\
=&\frac{1}{(2\pi i)^2} \int_{|\tilde{x}|=|\tilde{y}|=1} \log|P(a\tilde{x},b\tilde{y})|\frac{d\tilde{x}}{\tilde{x}}\frac{d\tilde{y}}{\tilde{y}}-\m_{a,b}(P^*)\\
=&2\log b-\frac{1}{2\pi} \int_{|\tilde{x}|=1, |\tilde{y}_1|\geq 1} \eta(\tilde{x},\tilde{y}_1)-\frac{1}{2\pi} \int_{|\tilde{x}|=1, |\tilde{y}_2|\geq 1} \eta(\tilde{x},\tilde{y}_2)\\
=&2\log b -\frac{1}{2\pi} \int_{|x|=a, |y_1|\geq b} \eta(x/a,y_1/b)-\frac{1}{2\pi} \int_{|x|=a, |y_2|\geq b} \eta(x/a,y_2/b).
\end{align*}
Now, each of the terms above can be further simplified in the following way. 
\begin{align*}
-\frac{1}{2\pi} \int_{|x|=a, |y_i|\geq b} \eta(x/a,y_i/b)=&-\frac{1}{2\pi} \int_{|x|=a, |y_i|\geq b} (\eta(x,y_i)-\eta(a,y_i)-\eta(x,b))\\
=&-\frac{1}{2\pi} \int_{|x|=a, |y_i|\geq b} (\eta(x,y_i)-\log (a)  d \arg y_i)-\log b
\end{align*}

In order to evaluate 
\begin{equation}\label{eq:regulator}
-\frac{1}{2\pi} \int_{|x|=a, |y_i|\geq b} \eta(x,y_i)
\end{equation}
subject to the condition $P(x,y_i)=0$, where $P(x,y_i)$ is a tempered polynomial,  we reduce to the classical case. 
In favorable cases, $\{|x|=a, |y_i|\geq b\}$ leads to a closed path which can be characterized as an element of 
 $H_1(E,\Z)^-$. 

The term 
\begin{equation}\label{eq:arg}
\log a \frac{1}{2\pi} \int_{|x|=a, |y_i|\geq b} d \arg y_i
\end{equation}
must also be considered for each case. If, as above, $\{|x|=a, |y_i|\geq b\}$ leads to a closed path, then this leads to a term that 
equals a multiple of $\log a$. 

Now that we have described the general situation, we will concentrate in the particular polynomials 
involved in Theorem \ref{bigthm}. 

\section{The connection between the regulator and the $L$-function for our polynomials} \label{sec:regulators-examples}

In this section we prove a relationship between regulators for $R_{-2}(x,y)$ and $S_{2,-1}(X,Y)$. 
The goal of this step is to relate the differential forms in the integral \eqref{eq:regulator}. This will allow us to use formula
\eqref{eq:Rogers-Zudilin} in order to evaluate those terms. A substantial part of what we present in this section 
was done by Touafek \cite{Touafek, Touafek-thesis}. We include it here for the sake of completeness. 

To make the notation easier to follow,
we write the variables of $S_{2,-1}$ with capital letters. 

The change of variables
\[\left\{\begin{array}{ll}
 X&=\displaystyle\frac{(\alpha+1)(x+y)}{x+y-\alpha}\\ \\
 Y&=\displaystyle\frac{(\alpha+1) ((\alpha-2)x-(\alpha+2)y)}{2(x+y-\alpha)}
\end{array}\right .\]
and
\begin{equation}\label{eq:change}
\left\{\begin{array}{ll}
x&=\displaystyle\frac{(\alpha+2) X+2Y}{2(X-(\alpha+1))}\\ \\
y&=\displaystyle\frac{(\alpha-2) X-2Y}{2(X-(\alpha+1))}
\end{array}\right .
\end{equation}
gives a birational transformation 
\begin{equation}\label{eq:varphi}
\varphi: R_\alpha(x,y) \rightarrow E_\alpha(X,Y)
\end{equation}
between
\[R_\alpha(x,y):=(1+x)(1+y)(x+y)-\alpha xy\]
and
\[E_\alpha(X,Y):=Y^2+2XY-\left(X^3 +\left(\frac{\alpha^2}{4} - \alpha - 3\right)X^2 +(\alpha+1)X\right).\]
Remark that this last equation yields $S_{2,-1}(X,Y)$ when $\alpha=-2$. We may continue this computation for arbitrary $\alpha$. 

The torsion group of $E_\alpha$ has order 6 and is generated by $P=\left(\alpha+1,\frac{(\alpha-2)(\alpha+1)}{2}\right)$, with $2P=\left(1,\frac{\alpha-2}{2}\right)$, $3P=(0,0)$, $4P=\left(1,-\frac{\alpha+2}{2}\right)$, $5P=\left(\alpha+1,-\frac{(\alpha+2)(\alpha+1)}{2}\right)$.

\begin{prop}\label{prop} We have the following relation in $\Z[E_{-2}(\Q)]^-$:
\[(X)\circ (Y)= -\frac{2}{3} (x\circ \varphi^{-1})\diamond (y\circ \varphi^{-1})\]
\end{prop}
\begin{proof}
We compute the divisors of some of the rational functions. 
\begin{align*}
(X)&=2(3P)-2O,\\
(Y)&=(3P)+(4P)+(5P)-3O \qquad(\mbox{for }\alpha=-2),\\
((\alpha+2) X+2Y)&=(3P)+(4P)+(5P)-3O,\\
((\alpha -2)X-2Y)&= (3P)+(2P)+(P)-3O,\\
(X-(\alpha+1))&= (P)+(5P)-2O \qquad(\mbox{for }\alpha=-2).
\end{align*}

Combining the above with the change of variables \eqref{eq:change}, we obtain, for $\alpha=-2$,
\begin{align*}
(x\circ \varphi^{-1})&= (3P)+(4P)-(P)-O,\\
(y\circ \varphi^{-1})&= (3P)+(2P)-(5P)-O.
\end{align*}

Finally, the diamond operations of $\{X,Y\}$ and $\{x\circ \varphi^{-1} ,y\circ \varphi^{-1} \}$ yield
\begin{align*}
(X)\diamond (Y)&=-4(P)-4(2P),\\
(x\circ \varphi^{-1})\diamond (y\circ \varphi^{-1})&=6(P)+6(2P),
\end{align*}
and this proves the desired identity. 
\end{proof}

\section{The integration path} \label{sec:path}

The goal of this section is to determine conditions for the integration paths in integrals \eqref{eq:regulator} and 
\eqref{eq:arg} corresponding to $S_{2,-1}(X,Y)$ and  $R_{-2}(x,y)$ to be closed paths. This will allow us to determine their homology class later.

We start with $S_{2,-1}(X,Y)$. By solving the equation
\[Y^2+2XY-X^3+X=0\]
for $Y$, we find the roots
\begin{equation} \label{eq:Y}
Y_\pm = -X\pm \sqrt{X^3+X^2-X}=X(-1\pm \sqrt{X+1-X^{-1}}),
\end{equation}
where we choose the branch of the square-root that is defined over $\C \setminus (-\infty, 0]$ and takes a number of argument $-\pi <\phi <\pi$ to a number of argument $\frac{\phi}{2}$.  

We wish to determine the values of $a>0$ for which the path 
\[|X|=a, \qquad |Y_\pm|\geq a\]
is a closed path. Namely, we seek to determine the $a$'s for which we always have $|Y_\pm|\geq a$ 
as long as $|X|=a$. 

\begin{lem}\label{lemma-path-S}
Let $t=a-a^{-1}$. Then we have the following. 
\begin{itemize}
\item $|X|=a, |Y_-|\geq a$ is a closed path for any $t \in \R$. 

\medskip

\item $|X|=a, |Y_+|\geq a$ is a closed path for any $|t|\geq 3$. 
\medskip

\item $|X|=a, |Y_+|\leq a$ is a closed path for any $|t|\leq 1$. 

\end{itemize}
\end{lem}

\begin{proof} Let 
\[\tilde{Y}_\pm = \frac{Y_\pm}{X},\]
and write $X=ae^{i \theta}$ with $-\pi \leq \theta < \pi$. 

Requesting that $|Y_\pm|\geq a$ for $|X|=a$ is equivalent to requesting that for all $-\pi \leq \theta < \pi$
\begin{align*}
|\tilde{Y}_{\pm}| =& |-1 \pm \sqrt{1+ae^{i\theta} - a^{-1}e^{-i\theta}} |\\
=& |-1\pm \sqrt{1+(a-a^{-1})\cos(\theta) +i(a+a^{-1})\sin(\theta)  }|\\
\geq& 1.
\end{align*}

Let
\[R:=|\sqrt{1+(a-a^{-1})\cos(\theta) +i(a+a^{-1})\sin(\theta)}|.\] 

By taking absolute values we can write 
\begin{align*}
R^2 =& \sqrt{( 1+(a-a^{-1})\cos(\theta) )^2 + ((a+a^{-1})\sin(\theta))^2   } \\
=& \sqrt{ 5\left(1+\frac{t^2}{4}\right) - 4\left(\cos(\theta) - \frac{t}{4}\right)^2 }\\
\geq & \sqrt{(|t|-1)^2},
\end{align*}
where the last inequality becomes an equality if $\theta=-\pi, 0$, according to the sign of $t$. Starting from now, we assume that $|t|\not = 1$ so that $R\not =0$. 

Let $\phi$ be such that 
\[\cos(\phi)=\frac{1+(a-a^{-1})\cos(\theta)}{R^2}, \qquad \sin(\phi)=\frac{(a+a^{-1})\sin(\theta)}{R^2}.\]
Notice that $\phi=\pm \pi$ only when $\theta=-\pi$ and $t>1$ or $\theta=0$ and $t<-1$. In these cases
\begin{align*}
|\tilde{Y}_{\pm}|=& |-1\pm\sqrt{1-|t|}|\\
=&\sqrt{1+(|t|-1)}\\
=&\sqrt{|t|}>1.
\end{align*}

Otherwise we have $-\pi <\phi<\pi$ 
and  this allows us to write
\begin{align*}
|\tilde{Y}_{\pm}| = &|-1 \pm R(\cos(\phi/2) +i\sin(\phi/2))|\\
= &\sqrt{1+  R^2 \mp 2R\cos(\phi/2))}\\
= &\sqrt{1+  R^2 \mp 2R\sqrt{\frac{1+\cos(\phi)}{2}}}\\
=&\sqrt{1+  R^2 \mp \sqrt{2(R^2+1+t\cos(\theta))}}.
\end{align*}

Now notice that 
\[|\tilde{Y}_{\pm}| \geq 1 \Longleftrightarrow R^2  \geq \pm\sqrt{2(R^2+1+t\cos(\theta))}.\]
From here, we get that 
\[|\tilde{Y}_{-}|\geq 1\]
for $|t|\not =1$, and this can be extended to the cases $|t|=1$ by continuity. 

We now look for conditions for having $|\tilde{Y}_{+}|\geq 1$. 

\begin{align}
|\tilde{Y}_{+}|\geq 1&\Longleftrightarrow R^2  \geq \sqrt{2(R^2 +1+t\cos(\theta) )} \nonumber \\
&\Longleftrightarrow R^4 \geq 2(R^2 +1+t\cos(\theta) )\nonumber \\
&\Longleftrightarrow 3 + t^2 - 4 \cos(\theta)^2 \geq 2 \sqrt{5 + t^2 - 4 \cos(\theta)^2 + 2t\cos(\theta) \label{ineq}}.
\end{align}
Consider $\theta=0$. Thus, $\cos(\theta)=1$ and we have
\begin{align*}
|\tilde{Y}_{+}|\geq 1&\Longleftrightarrow t^2 - 1  \geq 2 \sqrt{ t^2 + 2 t+ 1}\\
&\Longleftrightarrow  t \geq 3 \mbox{ or } t\leq -1.
\end{align*}
Now consider $\theta = -\pi$ and $\cos(\theta) = -1$. 
\begin{align*}
|\tilde{Y}_{+}|\geq 1&\Longleftrightarrow t^2 - 1  \geq 2 \sqrt{t^2-2t+1}\\
&\Longleftrightarrow  t \leq -3 \mbox{ or } t\geq 1.
\end{align*}
This explains the conditions in the statement of the Lemma. 

We will work now with general $\theta$. We have just seen that, for selective values of $\theta$, 
inequality \eqref{ineq} is true for $|t|\geq 3$, while the opposite inequality is true for $|t|<1$. We want to see that this is the case for any $-\pi \leq \theta <\pi$. 
The  case where the opposite inequality is always true certainly happens when the left-hand side is negative.
Therefore, we at most add more restrictions by squaring both sides. Thus consider 
\begin{align}
& \left(  3 + t^2 - 4 \cos(\theta)^2   \right)^2 -  4 \left(  5 + t^2 - 4 \cos(\theta)^2 + 2t\cos(\theta)   \right) \nonumber\\
=& 16\cos^4(\theta) -8\cos^2(\theta)(t^2+1) -8t\cos(\theta) + \left(t^4 +2t^2-11\right)\label{ineq2} \\
=& (4\cos^2(\theta)-(t^2+1))^2-12 -8t\cos(\theta)\nonumber
\end{align}
and request that this expression is non-positive for $|t|\geq 3$ and non-negative for $|t|\leq 1$.

Let $|t|\geq 3$. Then $t^2+1\geq 10 >4 \cos^2\theta$ and
\begin{align*}
 (4\cos^2(\theta)-(t^2+1))^2-12 -8t\cos(\theta) \geq &(4-(t^2+1))^2-12 -8|t|\\
\geq  & t^4-6t^2-3-8|t|\\
\geq & 0,
\end{align*}
where the last inequality was obtained by inspecting the roots of the polynomial.

Let $|t|\leq 1$. Let us consider \eqref{ineq2} as a polynomial in $x=\cos(\theta)$. 
 \begin{align*}
 f(x) &=      16x^4 -8x^2(t^2+1) -8tx + (t^4+2t^2-11)\\
  f'(x) &=      8\left(8x^3 -2x(t^2+1) -t \right) \\
  f''(x) &=      16\left(12x^2 -(t^2+1) \right)
 \end{align*}

Notice that the roots of $f''(x)$ are at $\pm \sqrt{\frac{t^2+1}{12}}$, and $f''(x)$ is negative in the open interval
between these roots and positive in the open set outside these roots.

Therefore, the maxima of $f(x)$ in $[-1,1]$ occur at either $1$,$-1$ or at points in $\left(-\sqrt{\frac{t^2+1}{12}}, \sqrt{\frac{t^2+1}{12}}\right)$.


Notice that for $|t|<1$,  $f(1)=(t-3)(t+1)^3<0$ and  $f(-1)=(t+3)(t-1)^3<0$. 


Let $x_M$ be a local maximum when $|t|<1$. We have
\begin{align*}
&3 + t^2 - 4 x_M^2  - 2 \sqrt{ 5 + t^2 - 4x_M^2 + 2tx_M  }\\ 
< & 3 + 1  - 2 \sqrt{  5  + t^2 - 4 \frac{t^2+1}{12} - 2|t|\sqrt{\frac{t^2+1}{12}} }\\
= & 4 -2  \sqrt{\frac{14}{3}+  \left( \frac{2t^2}{3} - |t|\sqrt{\frac{t^2+1}{3}}\right) }.
 \end{align*}


Let 
\[g(t):=\frac{2t^2}{3} - |t|\sqrt{\frac{t^2+1}{3}}.\]
Let us minimize $g(t)$ for $-1\leq t \leq 1$. Since $g(t)$ is an even function, we consider $0\leq t \leq 1$. 
We have
\[g'(t) = \frac{4t}{3}-\frac{2t^2+1}{\sqrt{3(t^2+1)}}.\]
We find that
\[g'(t)=0 \Longleftrightarrow t = \frac{1}{\sqrt{2}}.\]

Thus, we find and compare the values
\[g(1)=\frac{2-\sqrt{6}}{3}, \quad g\left(\frac{1}{\sqrt{2}}\right)=-\frac{1}{6}, \quad g(0)=0.\]
We find the minimum for  $0\leq t \leq 1$ at $t=\frac{1}{\sqrt{2}}$  and $g(t)\geq -\frac{1}{6}$ in $|t|\leq 1$. 
Putting this together, for $|t|<1$, we have
\begin{align*}
&3 + t^2 - 4 x_M^2  - 2 \sqrt{ 5 + t^2 - 4x_M^2 + 2tx_M  }\leq   4 -3\sqrt{2} <0,
\end{align*}
which proves that $f(x)< 0$ for $-1< t < 1$.


The cases $|t|=1$ are obtained by continuity. 
\end{proof}

We now proceed with the analysis of the integration paths corresponding to  $R_{-2}(x,y)$. The equation
\[(1+x)(1+y)(x+y)+2xy=0\]
can be rewritten as 
\[(x+1)y^2+(x^2+4x+1)y+(x^2+x)=0.\]
It is convenient to make the change of variables $x= x_1^2$ as well
as $y_1=y/x_1$. We then have
\begin{align}\label{eq:y1}
{y_1}_\pm = &\frac{-(x_1^2+4+x_1^{-2})\pm \sqrt{x_1^4+4x_1^2+10+4x_1^{-2}+x_1^{-4}}}{2(x_1+x_1^{-1})} \nonumber\\
=&\frac{-(2+(x_1+x_1^{-1})^2)\pm \sqrt{4+(x_1+x_1^{-1})^4}}{2(x_1+x_1^{-1})},
\end{align}
where, as before, the square-root is defined over $\C \setminus (-\infty, 0]$ and takes a number of argument $-\pi <\phi <\pi$ to a number of argument $\frac{\phi}{2}$.  

Following the previous notation, we set 
\begin{equation}\label{eq:y}
y_\pm = x_1{y_1}_\pm.
\end{equation}

We wish to determine the values of $a>0$ for which the path 
\[|x|=a^2, \qquad |y_\pm|\geq a\]
is a closed path. After the aforementioned change of variables, this result is equivalent to the path
\[|x_1|=a, \qquad |{y_1}_\pm|\geq 1\]
being closed.
\begin{lem} \label{lemma-path-R}
 Let $\sqrt{\frac{1+\sqrt{5}-\sqrt{2\sqrt{5}+2}}{2}}\leq a\leq \sqrt{\frac{1+\sqrt{5}+\sqrt{2\sqrt{5}+2}}{2}}$. Then 
 \[|x|=a^2, \qquad |y_-|\geq a\]
and
\[|x|=a^2, \qquad |y_+|\leq a\]
are closed paths.
\end{lem}

\begin{proof}
Write $x_1=ae^{i\theta}$ with $0\leq \theta <\pi$. For the particular case when $a=1$, 
we have 
\[{y_1}_\pm = \frac{-(1+2\cos^2\theta)\pm\sqrt{1+4\cos^4 \theta}}{2\cos \theta}.\]
Observe that $0<1+4\cos^4 \theta$ and both roots ${y_1}_\pm$ are real. Since $1+4\cos^4 \theta\leq (1+2\cos^2\theta)^2$,
\[|{y_1}_-|=\frac{1+2\cos^2\theta+\sqrt{1+4\cos^4 \theta}}{2|\cos \theta|}\geq \frac{\sqrt{1+4\cos^4 \theta}}{|\cos \theta|}=\sqrt{\frac{1}{\cos^2\theta}+4\cos^2\theta}> 1. \]
On the other hand, ${y_1}_+{y_1}_-=1$ implies that $|{y_1}_+|<1$.

From now on, we assume that $a>0$ and $a\not = 1$. Consider the case $\theta=0$. Then
\begin{align*}
|{y_1}_-| =& \frac{(2+(a+a^{-1})^2)+\sqrt{4+(a+a^{-1})^4}}{2(a+a^{-1})}\\
> & \frac{\sqrt{4+(a+a^{-1})^4}}{2(a+a^{-1})}\\
> & \frac{a+a^{-1}}{2}\\
> & 1.
\end{align*}
Since ${y_1}_+{y_1}_-=1$, this means that our goal is to find conditions on $a$ under which $|{y_1}_-|\geq 1$ for $0\leq \theta <\pi$ and we will automatically obtain $|{y_1}_+|\leq 1$. 

Consider the case of $\theta=\frac{\pi}{2}$. Then
\begin{align*}
|{y_1}_-| =&\left|\frac{(2+(ia-ia^{-1})^2)+\sqrt{4+(ia-ia^{-1})^4}}{2(ia-ia^{-1})}\right|\\
=&\left|\frac{(2-(a-a^{-1})^2)+\sqrt{4+(a-a^{-1})^4}}{2(a-a^{-1})}\right|
\end{align*}

Observe that $|2-(a-a^{-1})^2|\leq \sqrt{4+(a-a^{-1})^4}$. To see this, it suffices to 
take squares in both sides of the inequality. 

Then we can always write
\begin{align*}
|{y_1}_-| = &\frac{(2-(a-a^{-1})^2)+\sqrt{4+(a-a^{-1})^4}}{2|a-a^{-1}|}.
\end{align*}
We want to see under which conditions we have $|{y_1}_-|\geq 1$, which is equivalent to 
\begin{align*}
\sqrt{4+(a-a^{-1})^4} \geq & (a-a^{-1})^2+2|a-a^{-1}|-2.
\end{align*}
The above is always true for $1-\sqrt{3}< a-a^{-1}\leq \sqrt{3}-1$ because the right-hand side is negative. Otherwise, we square both sides and after simplification we obtain 
\begin{equation*}\label{sqrt{2}}
-\sqrt{2}\leq  a-a^{-1} \leq \sqrt{2},
\end{equation*}
or
\begin{equation}\label{sqrt{3}}
\frac{\sqrt{3}-1}{\sqrt{2}}< a <\frac{1+\sqrt{3}}{\sqrt{2}}.
\end{equation}

For general $\theta\not = 0, \frac{\pi}{2}$,  $a>0$ and $a\not =1$, our first step is to find conditions on $a$ so that the argument in the square root is 
never a real non-positive number. Then we consider 
\begin{align*}
\Delta=x_1^4+4x_1^2+10+4x_1^{-2}+x_1^{-4}=& (a^4+a^{-4})\cos(4\theta) +4 (a^2+a^{-2})\cos(2\theta) +10 \\&+ i ((a^4-a^{-4})\sin(4 \theta)+4(a^2-a^{-2})\sin(2 \theta)), 
\end{align*}
and we want to ensure that $\Delta \not \in (-\infty, 0]$.  For $\Delta$ to be a non-positive real number, we need
that the imaginary part be zero, namely,
\[\im(\Delta)=(a^4-a^{-4})\sin(4 \theta)+4(a^2-a^{-2})\sin(2 \theta)=0.\]
Since we work under the assumption that $a \not = 1$ and that $\sin(2\theta)\not =0$, 
we can divide the above identity by $(a^2-a^{-2})\sin(2\theta)$.  Then 
the above is equivalent to 
\begin{equation}\label{eq:cos}
\cos(2\theta)=-\frac{2}{a^2+a^{-2}}.
\end{equation}

Now we consider the real part of $\Delta$ under the above condition, 
\begin{align*}
\re(\Delta)=&  (a^4+a^{-4})\left(\frac{8}{(a^2+a^{-2})^2}-1\right) +2\\
=& -\frac{(a^4-2a^2-2-2a^{-2}+a^{-4})(a^4+2a^2-2+2a^{-2}+a^{-4})}{(a^2+a^{-2})^2}.
\end{align*}
Notice that $a^4+2a^2-2+2a^{-2}+a^{-4}=(a^2-a^{-2})^2+2(a^2+a^{-2})>0$ and similarly $(a^2+a^{-2})^2>0$. 
In order for $\Delta \in \C\setminus (-\infty, 0]$ we therefore need
\begin{equation}\label{eq:aa}
a^4-2a^2-2-2a^{-2}+a^{-4}<0.
\end{equation}
This happens when 
\[\frac{1+\sqrt{5}-\sqrt{2\sqrt{5}+2}}{2}<a^2<\frac{1+\sqrt{5}+\sqrt{2\sqrt{5}+2}}{2}.\]
Since we also assume that $a$ is positive, we will impose the condition
\begin{equation}\label{supercondition}
\sqrt{\frac{1+\sqrt{5}-\sqrt{2\sqrt{5}+2}}{2}}<a<\sqrt{\frac{1+\sqrt{5}+\sqrt{2\sqrt{5}+2}}{2}}.
\end{equation}
We remark that this condition implies condition \eqref{sqrt{3}}.

From now on we will assume \eqref{supercondition} and we will prove that $|{y_1}_-|\geq 1$ under this condition for any $0\leq \theta <\pi$. 

Notice that we proved that $|{y_1}_-|>1$ when $\theta=0, \frac{\pi}{2}$. If for some $\theta$ we have $|{y_1}_-|< 1$, then we must have $|{y_1}_-|=1$ at some intermediate point and we search for this point. 
If $|{y_1}_-|=1$, then we also have $|{y_1}_+|=1$ and
\[\left|2+(x_1+x_1^{-1})^2- \sqrt{4+(x_1+x_1^{-1})^4}\right|=\left|2+(x_1+x_1^{-1})^2+\sqrt{4+(x_1+x_1^{-1})^4}\right|.\]
An elementary computation shows that this can only happen when there is a $C \in \R$ such that
\[2+(x_1+x_1^{-1})^2 = iC\sqrt{4+(x_1+x_1^{-1})^4}.\]
Squaring both sides, we need
\[(2+(x_1+x_1^{-1})^2)^2 = -C^2(4+(x_1+x_1^{-1})^4),\]
or
\begin{equation}\label{eq:K}
(1+C^2)(4+(x_1+x_1^{-1})^4)+4(x_1+x_1^{-1})^2=0.
\end{equation}
Considering the imaginary part, we obtain
\[(1+C^2)((a^4-a^{-4})\sin(4\theta)+4 (a^2-a^{-2})\sin(2\theta))+4(a^2-a^{-2})\sin(2\theta)=0.\]
Notice that $\sin(2\theta)=0$ implies that $\theta=0$ or $\frac{\pi}{2}$. These cases were already discussed.
Similarly with $a=1$. In all other cases we divide the equation above by $2(a^2-a^{-2})\sin(2\theta)$. After some simplification, we obtain
\begin{equation}\label{eq:K2}
(1+C^2)=\frac{-2}{2+(a^2+a^{-2})\cos(2\theta)}.
\end{equation}
Replacing this in equation \eqref{eq:K} and taking the real part, we obtain,
\[(a^4+a^{-4})\cos(4\theta)+4(a^2+a^{-2})\cos(2\theta)+10=2(2+(a^2+a^{-2})\cos(2\theta))^2.\]
After further simplification, 
\[0=(a^4+a^{-4})+4\cos^2(2\theta)+4(a^2+a^{-2})\cos(2\theta)-2,\]
which implies
\[4=(a^2+a^{-2}+2\cos(2\theta))^2.\]
Since $a^2+a^{-2}\geq 2$, the number inside the parenthesis is positive, so we can write 
$2=a^2+a^{-2}+2\cos(2\theta)$ and further
\[4 \sin^2\theta=a^2+a^{-2}.\]
Thus, given $a$ with condition \eqref{supercondition}, we find at most two solutions $0<\theta<\pi$ to the equation above, with one solution in $(0,\frac{\pi}{2})$ and the other in $(\frac{\pi}{2}, \pi)$. Since we already verified that $|{y_1}_-|>1$ for $\theta=0, \frac{\pi}{2}$ and the case $\theta=-\pi$ is similar to the case $\theta=0$,
and in each interval we
have only once that $|{y_1}_-|=1$, we conclude that we never get $|{y_1}_-|<1$. 

Finally, we can extend condition \eqref{supercondition} to the extremes by continuity. 
\end{proof}

\section{The cycle of the integration path}\label{sec:cycle}

Now that we understand when the integration path in \eqref{eq:regulator} is a closed path, we need to 
 understand its class in the homology group $H_1(E,\Z)$. Let $\omega$ be the invariant holomorphic differential over the Weierstrass form determined by $S_{-2,1}(X,Y)=0$. 
 In order to compare and characterize the cycle $[\gamma]$ corresponding to the closed path $\gamma$, we compute $\int_\gamma \omega$.


\begin{lem}\label{lemma1} Let $a \in \R$ be such that  $\frac{\sqrt{5}-1}{2}\leq a \leq \frac{1+\sqrt{5}}{2}$. Then
\[\int_{|X|=a} \omega= - 2i \sqrt{ \frac{\sqrt5-1}{2} }K\left(i\left( \frac{\sqrt5-1}{2} \right)\right),\]
where the integral is performed over the path $|X|=a$, $|Y_-|\geq a$ and $Y_-$ is given by \eqref{eq:Y} and satisfies  $S_{-2,1}(X,Y_-)=0$, and
\[K(k):= \int_{0}^{\frac{\pi}{2}} \frac{d\theta}{\sqrt{1-k^2\sin^2\theta}}\]
is the complete Elliptic Integral of the First Kind. 
\end{lem}

\begin{proof}
First we assume that $\frac{\sqrt{5}-1}{2}< a < \frac{1+\sqrt{5}}{2}$. The extreme cases follow by continuity.  
By working with the Weierstrass form $S_{-2,1}(X,Y)$, we obtain the following.
\begin{align*}
 \int_{|X|=a} \omega =& \int_{|X|=a} \frac{d X}{2(Y_-+X)}\\
=& - \int_{|X|=a} \frac{d X}{2X\sqrt{X+1-X^{-1}}}\\
=& - \int_{-\pi}^{\pi} \frac{id\theta}{2\sqrt{ae^{i\theta}+1-a^{-1}e^{-i\theta}}}\\
=& - \int_{-\pi}^{\pi} \frac{id\theta}{2\sqrt{1+(a-a^{-1})\cos \theta +i(a+a^{-1})\sin \theta}}\\
=& - \int_{-\pi}^{0} \frac{id\theta}{2\sqrt{1+(a-a^{-1})\cos \theta +i(a+a^{-1})\sin \theta}}-\int_{0}^{\pi} \frac{id\theta}{2\sqrt{1+(a-a^{-1})\cos \theta +i(a+a^{-1})\sin \theta}}\\
=& - \int_{0}^{\pi} \frac{id\theta}{2\sqrt{1+(a-a^{-1})\cos \theta -i(a+a^{-1})\sin \theta}}-\int_{0}^{\pi} \frac{id\theta}{2\sqrt{1+(a-a^{-1})\cos \theta +i(a+a^{-1})\sin \theta}}\\
=& - i \re \int_{0}^{\pi} \frac{d\theta}{\sqrt{1+(a-a^{-1})\cos \theta +i(a+a^{-1})\sin \theta}}
\end{align*}
Our goal is to express the above integral in terms of elliptic integrals that we can easily characterize. In order to do this, 
we work with $\theta/2$ instead of $\theta$. 
\begin{align*}
&\int_{|X|=a} \omega\\
=& - i \re \int_{0}^{\pi} \frac{d\theta}{\sqrt{1+(a-a^{-1})(\cos^2 (\theta/2)-\sin^2 (\theta/2)) +i(a+a^{-1})2 \sin (\theta/2) \cos(\theta/2)}}\\
=& - i \re \int_{0}^{\pi} \frac{d\theta}{\cos (\theta/2)\sqrt{\frac{1}{\cos^2 (\theta/2)}+(a-a^{-1})\left(1-\tan^2 (\theta/2)\right) +i(a+a^{-1})2 \tan (\theta/2)}}\\
=& - i \re \int_{0}^{\pi} \frac{d\theta}{\cos (\theta/2)\sqrt{\tan^2 (\theta/2)+1+(a-a^{-1})\left(1-\tan^2 (\theta/2)\right) +i(a+a^{-1})2 \tan (\theta/2)}}\\
=& - i \re \int_{0}^{\pi} \frac{d\theta}{\cos (\theta/2)\sqrt{(1-a+a^{-1})\tan^2 (\theta/2)+1+(a-a^{-1}) +i(a+a^{-1})2 \tan (\theta/2)}}
\end{align*}
We make the change of variables $t=i\tan (\theta/2)$. This gives \[dt = \frac{id \theta}{2\cos^2 (\theta/2)} = \frac{i}{2}\left(1+\tan^2(\theta/2)\right) d\theta=\frac{1-t^2}{2} id\theta.\] 
Therefore, we have
\begin{align}
 \int_{|X|=a} \omega =&-2 i \re  \int_{0}^{i\infty} \frac{-idt }{\sqrt{(1-t^2)((1-a+a^{-1})(-t^2)+1+(a-a^{-1}) +(a+a^{-1})2 t)}}\nonumber\\
 = & - 2i  \im  \int_{0}^{i\infty} \frac{dt }{\sqrt{(1-a+a^{-1})(t^2-1)\left(t-\frac{\sqrt5+a+a^{-1}}{1-a+a^{-1}}\right) \left(t-\frac{-\sqrt5+a+a^{-1}}{1-a+a^{-1}}\right) }}\nonumber \\
 = & - 2i  \im  \lim_{R\rightarrow \infty}\int_{0}^{iR} \frac{dt }{\sqrt{(1-a+a^{-1})(t^2-1)\left(t-\frac{\sqrt5+a+a^{-1}}{1-a+a^{-1}}\right) \left(t-\frac{-\sqrt5+a+a^{-1}}{1-a+a^{-1}}\right) }}\label{integral}.
 \end{align} 
Since   $\frac{\sqrt{5}-1}{2}< a < \frac{1+\sqrt{5}}{2}$, we have $1-a+a^{-1}>0$. Notice that the polynomial inside the 
square-root has 4 real roots, 
\begin{equation}\label{eq:alpha}\alpha_1=\frac{\sqrt5 +a+a^{-1}}{1-a+a^{-1}},\, \alpha_2=1,\, \alpha_3=\frac{-\sqrt5 +a+a^{-1}}{1-a+a^{-1}},\, \alpha_4=-1,\end{equation} which satisfy
\[\alpha_1 >\alpha_2>0>\alpha_3 >\alpha_4.\]

In order to compute integral \eqref{integral}, we complete the vertical line with a horizontal line and a quarter of a circle of radius $R$. The integrand has no
poles in the interior of this region, but it has two poles on the segment $0\leq t \leq R$ when $R$ is sufficiently large. 
Thus, we modify the integration path by substracting a semicircle of radius  $\varepsilon$ around each pole.

Before proceeding any further,  we need the following result. 
  \begin{lem}\label{littlearc}
Let $P \in \C[t]$ be such that $P(0)\neq 0$, then \[\lim_{\varepsilon \rightarrow 0}\int_{-\varepsilon}^{\varepsilon} \frac{dt}{\sqrt{tP(t)}} = 0,\] 
where the integral is across a semicircle of radius $\varepsilon$ and center $t=0$. 
\end{lem}

\begin{proof} Since $P(0)\not = 0$, there is a real constant $C$ such that  $|P(t)|\geq C^2$ for $|t|\leq \varepsilon$. 
Setting $t= \varepsilon e^{i\theta}$, the absolute value of the integral is 
\[ \left|\int_{0}^{\pi} \frac{i\varepsilon e^{i\theta}d\theta}{\sqrt{\varepsilon P(\varepsilon e^{i\theta})}}\right|
\leq \frac{\sqrt \varepsilon }{C} \int_0^{\pi} d\theta = \frac{\pi \sqrt \varepsilon }{C}.\]

Since the radius $\varepsilon$ can be made arbitrarily small (without changing $C$), the limiting integral is 0.
\end{proof}

Using the previous lemma we write, for \[I= \frac{dt }{\sqrt{(t^2-1)\left(t-\frac{\sqrt5+a+a^{-1}}{1-a+a^{-1}}\right) \left(t-\frac{-\sqrt5+a+a^{-1}}{1-a+a^{-1}}\right) }},\]
\begin{align*}
\im \int_{0}^{iR}I 
= &-\im \int_{iR}^R I+\im \int_{0}^{1} I+\im \int_{1}^{\frac{\sqrt5 +a+a^{-1}}{1-a+a^{-1}}} I+\im\int_{\frac{\sqrt5 +a+a^{-1}}{1-a+a^{-1}}}^R I\\
= &-\im \int_{iR}^R I+\im \int_{1}^{\frac{\sqrt5 +a+a^{-1}}{1-a+a^{-1}}} I,
\end{align*}
where we have used that \[\im \int_{0}^{1} I=\im\int_{\frac{\sqrt5 +a+a^{-1}}{1-a+a^{-1}}}^R I=0\]
because the integrand is real in those intervals.

Notice that
\begin{align*}
\left|\int_{iR}^R I\right| = & \left|\int_{\frac{\pi}{2}}^0 \frac{ iRe^{i\theta}d\theta}{\sqrt{((Re^{i\theta})^2-1)\left(Re^{i\theta}-\frac{\sqrt5+a+a^{-1}}{1-a+a^{-1}}\right) \left(Re^{i\theta}-\frac{-\sqrt5+a+a^{-1}}{1-a+a^{-1}}\right) }}\right|\\
\ll \frac{1}{R},
\end{align*}
and therefore \[\int_{iR}^R I \rightarrow 0\quad \mbox{ as } \quad R\rightarrow \infty.\] 

Finally, we get 
\begin{align*}
 \int_{|X|=a} \omega =& -\frac{2 i}{\sqrt{1-a+a^{-1}}} \im \int_{1}^{\frac{\sqrt5 +a+a^{-1}}{1-a+a^{-1}}}  \frac{dt }{\sqrt{(t^2-1)\left(t-\frac{\sqrt5+a+a^{-1}}{1-a+a^{-1}}\right) \left(t-\frac{-\sqrt5+a+a^{-1}}{1-a+a^{-1}}\right) }}.\\
 \end{align*}
We need the following result. 
\begin{prop}[\cite{By}, formula 256.00 page 120] Let $\alpha_1>\alpha_2>\alpha_3>\alpha_4$ be real numbers and $\alpha_1\geq \gamma >\alpha_2$. Then
\begin{align*}
&\int_{\alpha_2}^{\gamma}\frac{dt}{\sqrt{(\alpha_1-t)(t-\alpha_2)(t-\alpha_3)(t-\alpha_4)}}\\ 
= &\frac{2}{\sqrt{(\alpha_1-\alpha_3)(\alpha_2-\alpha_4)}}F\left(\sin^{-1}\left(\sqrt{\frac{(\alpha_1-\alpha_3)(\gamma-\alpha_2)}{(\alpha_1-\alpha_2)(\gamma-\alpha_3)}}\right),\sqrt{\frac{(\alpha_1-\alpha_2)(\alpha_3-\alpha_4)}{(\alpha_1-\alpha_3)(\alpha_2-\alpha_4)}}\right),
\end{align*}
where 
\[F(\phi,k) = \int_{0}^{\phi} \frac{d\theta}{\sqrt{1-k^2\sin^2\theta}}.\]
\end{prop}

Recalling that  $K(k) = F\left(\frac{\pi}{2},k\right)$, and the values of the $\alpha_i$ given by \eqref{eq:alpha}, we have
 \begin{align*}
 \int_{|X|=a} \omega =& - \frac{2i}{\sqrt{(1-a+a^{-1})}} \frac{2}{\sqrt{(\alpha_1-\alpha_3)(\alpha_2-\alpha_4)}}K\left(\sqrt{\frac{(\alpha_1-\alpha_2)(\alpha_3-\alpha_4)}{(\alpha_1-\alpha_3)(\alpha_2-\alpha_4)}} \right) \\ 
 =&- \frac{2i}{\sqrt[4]{5}}K\left(\sqrt{\frac{\sqrt5-1}{2\sqrt5}}\right).
\end{align*}

By formula 160.02 on page 38 from \cite{By},
\[F(\phi,ik) = \frac{1}{\sqrt{1+k^2}} F\left(\sin^{-1}\left(\frac{ \sqrt{1+k^2}}{ \sqrt{1+k^2\sin^2\phi}}\sin\phi \right), \frac{k}{ \sqrt{1+k^2} }\right).\]
In particular, for $\phi=\frac{\pi}{2}$,
\[K(ik) = \frac{1}{\sqrt{1+k^2}} K\left(\frac{k}{ \sqrt{1+k^2} }\right).\]
Substituting $k$ by $ik$,
\[K(-k) = K(k) = \frac{1}{\sqrt{1-k^2}} K\left(\frac{ik}{ \sqrt{1-k^2} }\right).\]
Thus, 
\[- \frac{2i}{\sqrt[4]{5}}K\left(\sqrt{\frac{\sqrt5-1}{2\sqrt5}}\right) = - 2i \sqrt{ \frac{\sqrt5-1}{2} }K\left(i\left( \frac{\sqrt5-1}{2} \right)\right).\]
 
 \end{proof}
 \begin{lem}\label{lemma2} Let $a \in \R$ be such that $\sqrt{\frac{1+\sqrt{5}-\sqrt{2\sqrt{5}+2}}{2}}\leq a \leq \sqrt{\frac{1+\sqrt{5}+\sqrt{2\sqrt{5}+2}}{2}}$. Then 
\[\int_{\varphi_*(|x|=a^2)} \omega= - 2i \sqrt{ \frac{\sqrt5-1}{2} }K\left(i\left( \frac{\sqrt5-1}{2} \right)\right),\]
where the integral is performed over the path $|x|=a^2$, $|y_-|\geq a$,  $y_-$ is given by \eqref{eq:y1} and \eqref{eq:y} and satisfies 
$R_{-2}(x,y_-)=0$, and  $\varphi$ is given by \eqref{eq:varphi}. 
\end{lem}


\begin{proof}
As before, we assume that $\sqrt{\frac{1+\sqrt{5}-\sqrt{2\sqrt{5}+2}}{2}}< a< \sqrt{\frac{1+\sqrt{5}+\sqrt{2\sqrt{5}+2}}{2}}$. The extreme cases follow by continuity. Since we have
\begin{align*}
\int_{\varphi_*(|x|=a^2)} \omega=& \int_{|x|=a^2} \varphi^*\omega,
\end{align*}
we proceed to find $ \varphi^*\omega$.

By looking at the equations for $\varphi$, we have
\[dX = -\frac{2(dx+dy)}{(x+y+2)^2}.\]
By differentiating $R_{-2}(x,y)$, we get
\[(2x(y+1)+y^2+4y+1)dx+(2y(x+1)+x^2+4x+1)dy=0.\]
Thus, we obtain,
\[dX = \frac{2(y-x)dx}{(x+y+2)(2y(x+1)+x^2+4x+1)}.\]
Therefore $\varphi^* \omega$ is given by
\begin{align*}
\frac{dX}{2(X+Y)}=&-\frac{dx}{2y(x+1)+x^2+4x+1}.
\end{align*}
Since
\[y_\pm = x_1\frac{-(2+(x_1+x_1^{-1})^2)\pm \sqrt{4+(x_1+x_1^{-1})^4}}{2(x_1+x_1^{-1})},\]
where $x=x_1^2$, we have, when considering $y_-$,
\[\varphi^* \omega= -\frac{2dx_1}{x_1\sqrt{4+(x_1+x_1^{-1})^4}}.\]

Therefore, we find,
\begin{align*}
\int_{\varphi_*(|x|=a^2)} \omega=& -\int_{|x_1| = a} \frac{dx_1}{x_1\sqrt{4+(x_1+x_1^{-1})^4}}.
\end{align*}
We have omitted a factor of 2 because we are integrating in the full circle over $x_1$, which corresponds to twice the circle over $x$.

Setting $x_1=ae^{i\theta}$, we get,
\begin{align*}
\int_{\varphi_*(|x|=a^2)} \omega
=&-\int_{-\pi}^{\pi} \frac{i d\theta}{\sqrt{4+(ae^{i\theta}+a^{-1}e^{-i\theta})^4}}\\
=&-\int_{-\pi}^{0} \frac{i d\theta}{\sqrt{4+(ae^{i\theta}+a^{-1}e^{-i\theta})^4}}-\int_{0}^{\pi} \frac{i d\theta}{\sqrt{4+(ae^{i\theta}+a^{-1}e^{-i\theta})^4}}\\
=&-\int_{0}^{\pi} \frac{i d\theta}{\sqrt{4+(ae^{-i\theta}+a^{-1}e^{i\theta})^4}}-\int_{0}^{\pi} \frac{i d\theta}{\sqrt{4+(ae^{i\theta}+a^{-1}e^{-i\theta})^4}}\\
=&- 2 i \re \int_{0}^{\pi} \frac{ d\theta}{\sqrt{4+(ae^{i\theta}+a^{-1}e^{-i\theta})^4}}\\
=&- 2 i \re \int_{0}^{\frac{\pi}{2}} \frac{ d\theta}{\sqrt{4+(ae^{i\theta}+a^{-1}e^{-i\theta})^4}}- 2 i \re \int_{\frac{\pi}{2}}^\pi \frac{ d\theta}{\sqrt{4+(ae^{i\theta}+a^{-1}e^{-i\theta})^4}}\\
=&- 2 i \re \int_{0}^{\frac{\pi}{2}} \frac{ d\theta}{\sqrt{4+(ae^{i\theta}+a^{-1}e^{-i\theta})^4}}- 2 i \re \int_0^{\frac{\pi}{2}} \frac{ d\tau}{\sqrt{4+(ae^{i\tau}+a^{-1}e^{-i\tau})^4}}\\
=&- 4 i \re \int_{0}^{\frac{\pi}{2}} \frac{ d\theta}{\sqrt{4+(ae^{i\theta}+a^{-1}e^{-i\theta})^4}},
\end{align*}
where we did $\tau=\pi-\theta$. 

Let \[t=\frac{(ae^{i\theta}+a^{-1}e^{-i\theta})^2}{4}, \qquad dt =\frac{a^2e^{2i\theta}-a^{-2}e^{-2i\theta}}{2} id\theta=-2\sqrt{t(1-t)}d\theta.\]

This yields, 
\begin{align}\label{eq:realint}
\int_{\varphi_*(|x|=a^2)} \omega =& -\frac{i}{2} \re \int_{\frac{2 - (a^2 + a^{-2})}{4}}^{\frac{2 + (a^2 + a^{-2})}{4}}  \frac{dt}{\sqrt{t(1-t)(t^2 + 1/4 )}},
\end{align}
where the integral takes place over an arc connecting the two real points $\frac{2 - (a^2 + a^{-2})}{4}$ and $\frac{2 + (a^2 + a^{-2})}{4}$.
We close this arc with the segment of the real line connecting these two points. One can see that $|t|=\frac{a^2+a^{-2}+2\cos(2\theta)}{4}$. Choosing $\theta$ according to 
equation \eqref{eq:cos} so that $t$ is purely imaginary
shows that $|t|<\frac{1}{2}$ iff $a^4-2a^2-2-2a^{-2}+a^{-4}<0$, which is guaranteed by hypothesis (see condition \eqref{eq:aa}). 
Thus, the integrand has no poles in the interior of this region and
 the integral 
in \eqref{eq:realint} equals the integral over the corresponding real segment $\left[\frac{2 - (a^2 + a^{-2})}{4}, \frac{2 + (a^2 + a^{-2})}{4}\right]$. 
Since $(a^2+a^{-2}) \geq 2$ for any $a>0$, this segment always contains $[0,1]$. Moreover,
the polynomial under the square-root is positive only in $[0,1]$. Therefore, we get
\[\int_{\varphi_*(|x|=a^2)} \omega= -\frac{i}{2} \int_{0}^{1} \frac{dt}{\sqrt{t(1-t)(t^2 + 1/4 )}}.\]


Now we make the change of variables \[z=\frac{1-(1+\sqrt{5})t}{\left(1-\sqrt{5}\right)t-1}, \qquad t=\frac{(1+z)}{(1-\sqrt{5})z+(1+\sqrt{5})},\qquad dt = \frac{2\sqrt{5} dz}{((1-\sqrt{5})z+(1+\sqrt{5}))^2}.\]

We obtain
\begin{align*}
\int_{\varphi_*(|x|=a^2)} \omega=& -i\sqrt{\frac{\sqrt{5}-1}{2}}\int_{-1}^1 \frac{dz}{\sqrt{\left(1-z^2\right)\left(1+\left(\frac{1-\sqrt{5}}{2}\right)^2z^2\right)}}\\
=&-2i\sqrt{\frac{\sqrt{5}-1}{2}}\int_{0}^1 \frac{dz}{\sqrt{\left(1-z^2\right)\left(1+\left(\frac{1-\sqrt{5}}{2}\right)^2z^2\right)}}\\
=& -2i\sqrt{\frac{\sqrt{5}-1}{2}} K\left(i\left(\frac{\sqrt{5}-1}{2}\right) \right).
\end{align*}
\end{proof}
 
Lemmas \ref{lemma1} and \ref{lemma2} imply that the integration paths yield the same class in the homology, and this is also independent of the value of the parameter $a$ as long as $a$ satisfies the conditions that we discovered in Section \ref{sec:path}. 
Finally, we remark that the resulting integrals are purely imaginary, showing that in fact, the homology class lies in $H_1(E,\Z)^-$, which is consistent with the discussion in Remark \ref{remark}.

 \section{The integral over $\arg y$}\label{sec:arg}
In this section we compute the integrals \eqref{eq:arg}.
\begin{lem} \label{lemma15} Let $Y_-$ be the root of $S_{2,-1}(X,Y)=0$ defined by \eqref{eq:Y} and let
$a \in \R$ be such that  $\frac{\sqrt{5}-1}{2}\leq a \leq \frac{1+\sqrt{5}}{2}$.
Then
\[\frac{1}{2\pi} \int_{|X|=a} d \arg Y_-=1.\]
 \end{lem}

\begin{proof} As usual, assume that  $\frac{\sqrt{5}-1}{2}< a < \frac{1+\sqrt{5}}{2}$. The extreme cases follow by continuity.
First recall that $Y_-=X(-1-\sqrt{X+1-X^{-1}})$ and 
\begin{align}\label{differential}
\frac{1}{2\pi}\int_{|X|=a}d\arg Y_-=&\frac{1}{2\pi i }\int_{|X|=a}\frac{dX}{X}+\frac{1}{2\pi i }\int_{|X|=a} \frac{d(1+\sqrt{X+1-X^{-1}})}{1+\sqrt{X+1-X^{-1}}}\\
=&\frac{1}{2\pi i }\int_{|X|=a}\frac{dX}{X}+  \frac{1}{2\pi }\im \int_{|X|=a}\frac{X+X^{-1}}{f(X-X^{-1})}\frac{dX}{X}, \nonumber
\end{align}
where
\[f(t)=2(t+1+\sqrt{t+1}).\]

It is clear that \[\frac{1}{2\pi i }\int_{|X|=a}\frac{dX}{X}=1.\]

We need to prove that the second term in \eqref{differential} equals zero. Let 
\[I(a)=\frac{1}{2\pi }\im \int_{|X|=a}\frac{X+X^{-1}}{f(X-X^{-1})}\frac{dX}{X} = \frac{1}{2\pi} \re \int_0^{2\pi} \frac{ae^{i \theta}+a^{-1}e^{-i\theta}}{f(ae^{i \theta}-a^{-1}e^{-i\theta})}d\theta.\]

By setting $\tau=\pi-\theta$,
\[I(a)= \frac{1}{2\pi} \re \int_{-\pi}^{\pi} \frac{-ae^{-i \tau}-a^{-1}e^{i\tau}}{f(-ae^{-i \tau}+a^{-1}e^{i\tau})}d\tau=-I(a^{-1}).\]
In particular,
\[I(1)=0.\]

Recall that we are choosing $a$ such that $-1<a-a^{-1}<1$. This guarantees that 
\[X+1-X^{-1} \not \in (-\infty, 0] \mbox{ as long as } |X|=a.\]
Therefore we can take a branch of the square root and further a branch of the logarithm so that 
$g(X)=\log(1+\sqrt{X+1-X^{-1}})$ is well-defined and holomorphic in an open ring containing $|X|=a$ and $|X|=1$. 
This shows that $I(a)=I(1)=0$ for $a$ such that $-1<a-a^{-1}<1$.

\end{proof}

\begin{lem} \label{lemma16} Let $y_-$ be the root of $R_{-2}(x,y)=0$ defined by \eqref{eq:y1} and \eqref{eq:y}, and let
$a \in \R$ be such that $\sqrt{\frac{1+\sqrt{5}-\sqrt{2\sqrt{5}+2}}{2}}\leq a \leq \sqrt{\frac{1+\sqrt{5}+\sqrt{2\sqrt{5}+2}}{2}}$.
Then
\[\frac{1}{2\pi} \int_{|x|=a^2} d \arg y_-=\frac{1}{2}.\]
 \end{lem}

\begin{proof}
As usual, assume that  $\sqrt{\frac{1+\sqrt{5}-\sqrt{2\sqrt{5}+2}}{2}}<a<\sqrt{\frac{1+\sqrt{5}+\sqrt{2\sqrt{5}+2}}{2}}$. The extreme cases follow by continuity.

Write as before $x=x_1^2$ and 
\[y_-=x_1\frac{-(2+(x_1+x_1^{-1})^2)-\sqrt{4+(x_1+x_1^{-1})^4}}{2(x_1+x_1^{-1})}.\]
Then we have
\begin{align}
 \frac{1}{2\pi}\int_{|x|=a^2}d\arg y_-=& \frac{1}{4\pi}\int_{|x_1|=a}d\arg y_-\nonumber\\
 =&\frac{1}{4\pi i} \int_{|x_1|=a} \frac{d x_1}{x_1} +\frac{1}{4\pi i} \int_{|x_1|=a}\frac{d\left(2+(x_1+x_1^{-1})^2+\sqrt{4+(x_1+x_1^{-1})^4}\right)}{2+(x_1+x_1^{-1})^2+\sqrt{4+(x_1+x_1^{-1})^4}}\label{crazyintegral}\\
 &-\frac{1}{4\pi i} \int_{|x_1|=a}\frac{d(x_1+x_1^{-1})}{x_1+x_1^{-1}}.\nonumber
 \end{align}
Notice that
\[\frac{1}{2\pi i} \int_{|x_1|=a}\frac{d(x_1+x_1^{-1})}{x_1+x_1^{-1}}=\frac{1}{2\pi i} \int_{|x_1|=a}\frac{x_1-x_1^{-1}}{x_1+x_1^{-1}}\frac{d x_1}{x_1}.\]
It can be proven that the above integral is zero by the same reasoning that we did in Lemma \ref{lemma15}. Similarly we conclude that the second integral in \eqref{crazyintegral} is zero as well. 

We get
\[ \frac{1}{2\pi}\int_{|x|=a^2}d\arg y_-=\frac{1}{4\pi i} \int_{|x_1|=a} \frac{d x_1}{x_1}=\frac{1}{2}.\]
\end{proof}
 
\section{The proof of Theorem \ref{bigthm}} \label{sec:proof}

We have now all the elements to prove Theorem \ref{bigthm}.

First consider the case of $S_{2,-1}(X,Y)$. By Lemma \ref{lemma-path-S} from Section \ref{sec:path}, there are two cases where the integration paths are closed. Either $|t|\leq 1$, which implies $|a-a^{-1}|\leq 1$
and  $a \in \left[\frac{\sqrt{5}-1}{2}, \frac{1+\sqrt{5}}{2}\right]$, or $|t|\geq 3$, which implies $a-a^{-1}\geq 3$ and $a\geq \frac{3+\sqrt{13}}{2}$
or $a-a^{-1}\leq -3$ and $0\leq a\leq \frac{-3+\sqrt{13}}{2}$.

For $a \in \left[\frac{\sqrt{5}-1}{2}, \frac{1+\sqrt{5}}{2}\right]$, by Lemma \ref{lemma-path-S},
$|Y_-|\geq a$ and $|Y_+|\leq a$. Following the discussion from Section \ref{sec:arbitrary-torus}, we have
\begin{align*}
 \m_{a,a}(S_{2,-1})=&2\log a  -\frac{1}{2\pi}\int_{|X|=a}(\eta(X,Y_-)-\log(a) d\arg Y_-) -\log a.
 \end{align*}

Lemma \ref{lemma1} from Section \ref{sec:cycle} implies that 
\[-\frac{1}{2\pi}\int_{|X|=a}\eta(X,Y_-)=-\frac{1}{2\pi}\int_{|X|=1}\eta(X,Y_-).\]
We also have by combining Lemma \ref{lemma2} and Proposition \ref{prop},
\[-\frac{1}{2\pi}\int_{|X|=1}\eta(X,Y_-)=-\frac{2}{3}\frac{1}{2\pi} \int_{\varphi_*(|x|=1)} \eta(x\circ \varphi^{-1},y_-\circ \varphi^{-1}).\]
The result of Rogers and Zudilin \eqref{eq:Rogers-Zudilin} implies
\[-\frac{1}{2\pi} \int_{\varphi_*(|x|=1)} \eta(x\circ \varphi^{-1},y_-\circ \varphi^{-1})=3L'(E_{20},0).\]
By combining all the above equations, we finally obtain
\[-\frac{1}{2\pi}\int_{|X|=a}\eta(X,Y_-)=2L'(E_{20},0).\]

On the other hand, by Lemma \ref{lemma15} from Section \ref{sec:arg},
\[\frac{1}{2\pi}\int_{|X|=a}d\arg Y_- = 1.\]

Combining all of the above,
\begin{align*}
 \m_{a,a}(S_{2,-1})=&2 \log a + 2L'(E_{20},0).
\end{align*} 

When $a\geq \frac{3+\sqrt{13}}{2}$ or $a\leq \frac{-3+\sqrt{13}}{2}$,  Lemma \ref{lemma-path-S} implies that we have 
$|Y_\pm|\geq a$. Following the discussion from Section \ref{sec:arbitrary-torus}, we have
\begin{align*}
 \m_{a,a}(S_{2,-1})=&2\log a -\frac{1}{2\pi}\int_{|X|=a}(\eta(X,Y_-)-\log(a) d\arg Y_-) -\log a\\
 &-\frac{1}{2\pi}\int_{|X|=a}(\eta(X,Y_+)-\log(a) d\arg Y_+) -\log a\\
 =&-\frac{1}{2\pi}\int_{|X|=a}(\eta(X,X^3-X)-\log(a) d\arg (X^3-X)) \\
 =& -\frac{1}{2\pi}\int_{|X|=a}(\log|X|d\arg (X^3-X) - \log|X^3-X|d \arg X -\log(a) d\arg (X^3-X)) \\
 =& \m_{a}(X^3-X)\\
 =& \m(a^3X^3-aX),
 \end{align*}
where we have used that $Y_+Y_-=-X^3+X$. 

On the one hand, when $a\geq \frac{3+\sqrt{13}}{2}>1$, we have that 
\[\m(a^3X^3-aX)=3\log a +\m(X^2-1/a^2)=3 \log a.\]

On the other hand, when $0<a\leq \frac{-3+\sqrt{13}}{2}<1$, we have that
\[\m(a^3X^3-aX)=3\log a +\m(X^2-1/a^2)=3 \log a+2\log a^{-1}=\log a.\]


Now we work with $R_{-2}(x,y)$. By Lemma \ref{lemma-path-R}, the only case that we can consider is $\sqrt{\frac{1+\sqrt{5}-\sqrt{2\sqrt{5}+2}}{2}}\leq a \leq \sqrt{\frac{1+\sqrt{5}+\sqrt{2\sqrt{5}+2}}{2}}$, where we recall that we have $|x|=a^2$
while $|y_-|\geq a$ and $|y_+|\leq a$. By the discussion from Section \ref{sec:arbitrary-torus}, we have
\begin{align*}
\m_{a^2,a}(R_{-2})-\m_{a^2}(x+1)=&2\log a -\frac{1}{2\pi} \int_{|x|=a^2, |y_-|\geq a} (\eta(x,y_-)-  2\log(a) d\arg y_-) -\log a. 
\end{align*}
By combining Lemma \ref{lemma2} from Section \ref{sec:cycle} together with formula \eqref{eq:Rogers-Zudilin}, we can write
\[-\frac{1}{2\pi}\int_{|x|=a^2}\eta(x,y_-)=-\frac{1}{2\pi} \int_{|x|=1} \eta(x,y_-)=3L'(E_{20},0).\]

On the other hand, by  Lemma \ref{lemma16} from Section \ref{sec:arg},
\[\frac{1}{2\pi}\int_{|x|=a^2}d\arg y_- = \frac{1}{2}.\]

Finally, for $a\geq 1$,
\[\m_{a^2}(x+1)=\m(a^2x+1)=2\log a,\]
while for $a\leq 1$,
\[\m_{a^2}(x+1)=\m(a^2x+1)=0.\]

Putting everything together,
\begin{align*}
\m_{a^2,a}(R_{-2})=&2\log a +2 \log \max\{a,1\} +3 L'(E_{20},0).
\end{align*}

This completes the proof of our main result. 

\section{Conclusion}
There are several directions for further exploration. The most immediate question that we have is  
the completion of the statement of Theorem \ref{bigthm}, in the sense that we would like to give formulas for 
$\m_{a,b}(S_{2,-1})$ and $\m_{a,b}(R_{-2})$ for any positive parameters $a$ and $b$. This is a challenging problem, as it requires to integrate $\eta(x,y)$ in a path that is not closed and cannot be easily identified as a cycle in the homology group.

A different direction would be to consider other polynomials from Boyd's families. 

Finally, it would be also natural to explore this new definition of Mahler measure over arbitrary tori for arbitrary polynomials in a more general context
and to relate it to other constructions, such as the Ronkin function associated to amoebas (see \cite{amoeba} for further details). 

\section*{Acknowledgments}
We are grateful to Marie-Jos\'e Bertin for providing us a copy of Touafek's doctoral thesis \cite{Touafek-thesis}.

\end{document}